\documentclass{siamltex}
\usepackage{fullpage,enumerate,setspace,changepage,multirow,rotating,xcolor}
\usepackage{graphicx,amsmath,amssymb,amsfonts,url, algorithm, algorithmic,}
\usepackage[small]{caption}
\def\grad{\nabla}

\def\bo{\mathbf{0}}

\def\cA{\mathcal{A}}
\def\cB{\mathcal{B}}
\def\cC{\mathcal{C}}
\def\cD{\mathcal{D}}

\def\cF{\mathcal{F}}

\def\cK{\mathcal{K}}
\def\cL{\mathcal{L}}

\def\cN{\mathcal{N}}
\def\cO{\mathcal{O}}
\def\cP{\mathcal{P}}
\def\cQ{\mathcal{Q}}

\def\cS{\mathcal{S}}

\def\eq{\[}
\def\en{\]}

\def\smskip{\smallskip}

\def\texitem#1{\par\smskip\noindent\hangindent 25pt
               \hbox to 25pt {\hss #1 ~}\ignorespaces}

\def\norm#1{\|#1\|}

\newcommand{\BEAS}{\begin{eqnarray*}}
\newcommand{\EEAS}{\end{eqnarray*}}
\newcommand{\BEA}{\begin{eqnarray}}
\newcommand{\EEA}{\end{eqnarray}}
\newcommand{\BEQ}{\begin{eqnarray}}
\newcommand{\EEQ}{\end{eqnarray}}
\newcommand{\BIT}{\begin{itemize}}
\newcommand{\EIT}{\end{itemize}}
\newcommand{\BNUM}{\begin{enumerate}}
\newcommand{\ENUM}{\end{enumerate}}

\newcommand{\BA}{\begin{array}}
\newcommand{\EA}{\end{array}}

\newcommand{\ones}{\mathbf 1}

\newcommand{\reals}{\mathbb{R}}
\newcommand{\integers}{\mathbb{Z}}

\newcommand{\Rank}{\mathop{\bf rank}}

\newcommand{\argmin}{\mathop{\rm argmin}}
\newcommand{\argmax}{\mathop{\rm argmax}}

\newcommand{\dom}{\mathop{\bf dom}}

\newcommand{\relint}{\mathop{\bf rel int}}

\newif\ifpagenumbering
\pagenumberingtrue

\pagenumberingfalse

\newsavebox{\remarkbox}
\savebox{\remarkbox}{\noindent\bf Remark}
\def\fprod#1{\left\langle#1 \right\rangle}

\def\xik{\xi^{(k)}}

\def\fprod#1{\left\langle#1\right\rangle}
\def\alcc{ALCC algorithm}
\def\apg{APG algorithm}
\def\oracle{\mbox{\scshape{Oracle}}}

\def\xik{\xi_k}
\def\xikstar{\xi_k^\ast}

\newtheorem{assumption}[theorem]{Assumption}

\title{An Augmented Lagrangian Method for Conic Convex Programming}
\author{N. S. Aybat\thanks{IE Department, Pennsylvania State University.
    Email: {\tt nsa10@psu.edu}} \and G. Iyengar\thanks{IEOR
    Department, Columbia University.
    Email: {\tt gi10@columbia.edu}}}
\begin{document}
\maketitle
\begin{abstract}
We propose a new first-order augmented Lagrangian algorithm ALCC for solving
 convex conic programs of the form
\eq
\min \big\{ \rho(x) + \gamma(x) : Ax-b\in\cK,~x \in \chi\big\},
\en
where 
$\rho:\reals^n\rightarrow\reals\cup\{+\infty\}$,
$\gamma:\reals^n\rightarrow\reals$ are closed, convex
functions, and $\gamma$ has a Lipschitz continuous gradient,
$A\in\reals^{m\times n}$, $\cK\subset\reals^m$ is a
closed convex cone, and $\chi\subset\dom(\rho)$ is a ``simple'' convex
compact set 
such that optimization problems of the form $\min\{\rho(x) +
\norm{x-\bar{x}}_2^2: x \in \chi\}$ can be efficiently solved.
We show that any limit point of the primal ALCC iterates 
is an
optimal solution of the conic convex problem, and the dual ALCC iterates
have a unique limit point that is a Karush-Kuhn-Tucker~(KKT) point of the conic
program. We also show that for
{\em any} $\epsilon>0$, the primal ALCC iterates are $\epsilon$-feasible and
$\epsilon$-optimal after $\cO(\log(\epsilon^{-1}))$ iterations which
require solving $\cO(\epsilon^{-1}\log(\epsilon^{-1}))$ 
problems of
the form $\min_x\{\rho(x)+\norm{x-\bar{x}}_2^2:\ x\in\chi\}$.
\end{abstract}
\section{\normalsize Introduction}
In this paper we propose an inexact augmented Lagrangian algorithm~(ALCC)
for solving conic convex problems of the form
\begin{equation}
\label{eq:conic_problem}
(P):  \min \big\{\rho(x) + \gamma(x) :  Ax-b\in\cK,~  x \in \chi\big\},
\end{equation}
where 
$\rho:\reals^n\rightarrow\reals\cup\{+\infty\}$,
$\gamma:\reals^n\rightarrow\reals$ are proper, closed, convex
functions, and $\gamma$ has a  Lipschitz continuous gradient $\grad
\gamma$ with the Lipschitz constant $L_{\gamma}$, $A\in\reals^{m\times
  n}$, $\cK\subset\reals^m$ is a
nonempty, closed, convex cone, and $\chi\subset\dom(\rho)$ is a ``simple''
compact set in the sense that the optimization problems of the form
\begin{equation}
\label{eq:nonsmooth-operation}
\min_{ x \in \chi} \left\{\rho(x)+\norm{x-\bar{x}}_2^2\right\}
\end{equation}
can be efficiently solved for any  $\bar{x}\in\reals^n$. Note
that we do not require $A\in\reals^{m\times n}$  to satisfy any
additional regularity properties.
For notational convenience, we set
\eq
p(x) := \rho(x) + \gamma(x).
\en
In some problems, the compact set $\chi$ is explicitly present. For example, in a zero-sum game the decision $x$ represents a mixed strategy and the set $\chi$ is a simplex. In others, $\chi$ may not be explicitly present, but one can formulate an equivalent problem where the vector of decision variables can be constrained to lie in a bounded feasible set without any loss of generality. For example, if $\gamma$ is strongly convex, or if $\rho$ is a norm and $\gamma(\cdot)\geq 0$, then the decision vector $x$ can be restricted to lie in a appropriately defined norm ball centered at any feasible solution.

We assume that the following constraint qualification holds for $(P)$.
\begin{assumption}
\label{asp:KKT}
The problem $(P)$ in \eqref{eq:conic_problem} has a Karush-Kuhn-Tucker~(KKT) point,
i.e., there  exists $y^*\in\cK^*$ such that
$g_0(y^*):=\inf\{p(x)-\fprod{y^*,~Ax-b}:\ x\in\chi\} =p^* > -\infty$,
where $p^\ast$ denotes the optimal value of $(P)$ and $\cK^\ast$
denotes the dual cone corresponding to $\cK$, i.e., $\cK^\ast:=\{y\in\reals^m:\ \fprod{y,x}\geq 0\ \forall x\in\cK\}$.
\end{assumption}\\
Assumption~\ref{asp:KKT} clearly holds whenever there exists
$\tilde{x}\in\relint(\chi)$ such that $A\tilde{x}-b\in
\mathbf{int}(\cK)$~\cite{Boyd04_1B}.

\subsection{Special cases}
Many important optimization problems are special cases of
\eqref{eq:conic_problem}. Below, we briefly discuss some
examples.

\textbf{\emph{Min-max games with convex loss function:}}
  This problem is a generalization of the matrix game discussed in
  \cite{Nesterov05}. The decision maker can choose from $n$ possible
  actions. Let $x \in \reals^n_{+}$ denote a mixed strategy over the set of
  actions, i.e., $x\in\chi := \{x: \sum_{j=1}^n x_j = 1, x \geq \bo\}$. Suppose the
  mixed strategy $x$ must
  satisfy constraints of the form $Ax - b \in \cK$. These constraints
  could be modeling average cost constraints. For example, one may
  have  constraints of the form $Ax \leq b$, where $A \in \reals^{m\times n}$  and $A_{ij}$ denotes
  amount of resource $i$ consumed by action $j$. One may also have
  constraints that restrict the total probability weight of some given subsets
  of actions.

  The adversary has $p$ possible actions. The expected loss to decision maker when she
  chooses the mixed strategy $x\in\reals^n$ and the adversary chooses the mixed
  strategy $y\in \reals^p$ is given by
  \eq
  \rho(x) + y^T Cx - \phi(y),
  \en
  where $\rho$ is a convex function, and $\phi$ is a strongly convex
  function. Then the decision maker's optimization problem that minimizes the expected worst case loss is given by
  \begin{equation}
    \label{eq:min-max-strat}
    \min \left\{ \rho(x) + \gamma(x):\ Ax - b \in \cK,\ x \in \chi\right\},
  \end{equation}
  where
  \begin{equation}
    \label{eq:matrix-game}
    \gamma(x) = \max\bigg\{ y^T Cx -\phi(y):\ \sum_{k=1}^py_k = 1,\ y \geq \bo\bigg\}.
  \end{equation}
  From Danskin's theorem, it follows that
  $\grad\gamma(x)=C^Ty(x)$, where $y(x)$ denotes the unique minimizer in
  \eqref{eq:matrix-game} for a given $x$.
  In \cite{Nesterov05}, Nesterov
  showed that $\grad\gamma$ is Lipschitz continuous with Lipschitz
  constant $\sigma_{\max}(C)^2/\tau$, where $\tau$ denotes the convexity
  parameter for the strongly convex function $\phi$. Thus, it follows that
  the minimax optimization problem~\eqref{eq:min-max-strat} is a special
  case of \eqref{eq:conic_problem}.

\textbf{\emph{Problems with semidefinite constraints:}}
Let $\cS^m$ denote the set of $m \times m$ symmetric matrices, and
let $\cS^m_+$ denote the closed convex cone of $m \times m$ symmetric positive
semidefinite matrices.
A convex optimization problem with a linear matrix inequality constraint is of the form
\begin{equation}
  \label{eq:cvx-sdp}
  \min \bigg\{ \rho(x): \sum_{j=1}^n A_j x_j + B \in \cS^m_+\bigg\},
\end{equation}
where $\rho$ is a convex function, $B \in \cS^m$, and $A_j \in \cS^m$
for  $j = 1, \ldots, n$.
Convex problems of the form~\eqref{eq:cvx-sdp} can model many applications in
engineering, statistics
and combinatorial optimization
\cite{Boyd04_1B}. In most of these applications, either the
constraints imply that the decision vector $x$ is bounded, or one can often
establish that the optimal solution lies in a norm-ball. In such cases,
\eqref{eq:cvx-sdp} is a special case of \eqref{eq:conic_problem}.
Consider the $\ell_1$-minimization problem of the form
\begin{equation}
  \label{eq:min-l1-LMI}
  \min\bigg\{\norm{x}_1: \sum_{j=1}^n A_j x_j + B \in \cS^m_+\bigg\}.
\end{equation}
Suppose a feasible solution $x_0$ for this problem is known.
Then~\eqref{eq:min-l1-LMI} is a special case of \eqref{eq:conic_problem}
with $\rho(x)=\norm{x}_1$,
$\gamma(\cdot)= 0$, $\cK=\cS^m_+$ and
$\chi=\{x\in\reals^n:\norm{x}_1\leq\norm{x_0}_1\}$.
The main bottleneck step in solving this problem using the \alcc\  
reduces to the ``shrinkage'' problem of the form
$\min\{\lambda\norm{x}_1+\norm{x-\bar{x}}_2^2:\ \norm{x}_1\leq
\norm{x_0}_1\}$
that can be solved very efficiently for any given $\bar{x}\in\reals^n$ and $\lambda>0$.
\subsection{Notation}
Let $S\subset\reals^m$ be a nonempty, closed, convex
set. Let $d_S:\reals^m\rightarrow\reals_+$ 
denote the function
\begin{equation}
\label{eq:dist_func}
d_S(\bar{x}):=\min_{x \in S} \norm{x-\bar{x}}_2,
\end{equation}
i.e., $d_{S}(\bar{x})$ denotes the $\ell_2$-distance of the vector
$\bar{x}\in\reals^m$ to the set $S$. Let
\begin{equation}
  \label{eq:Kproj}
  \Pi_S(\bar{x}):=\argmin\{\norm{x-\bar{x}}_2:\ x\in S\},
\end{equation}
denote the $\ell_2$-projection of the vector $\bar{x}\in\reals^m$ onto the set
$S$. Since $S\subset\reals^m$ is a nonempty, closed, convex set,
$\Pi_S(\cdot)$ is well defined. Moreover, $d_{S}(\bar{x}) =
\norm{\bar{x}-\Pi_{S}(\bar{x})}_2$.
\subsection{New results}
The main results of this paper are as follows:
\begin{enumerate}[(a)]
\item Every limit point of the sequence of  ALCC primal iterates $\{x_k\}$
  is an optimal solution of \eqref{eq:conic_problem}.
\item The sequence of ALCC dual iterates $\{y_k\}$ converges to a KKT point of
  \eqref{eq:conic_problem}.
\item For \emph{all} $\epsilon>0$, the primal ALCC iterates $x_k$ are
  $\epsilon$-feasible, i.e., $x_k\in\chi$ and $d_\cK(Ax_k-b)\leq\epsilon$,
  and $\epsilon$-optimal, i.e., $|p(x_k)-p^*|\leq\epsilon$ after at most
  $\cO\left(\log\left(\epsilon^{-1}\right)\right)$ ALCC
  iterations that require solving at most
  $\cO(\epsilon^{-1}\log(\epsilon^{-1}))$ problems of the
  form \eqref{eq:nonsmooth-operation}.
\end{enumerate}
\smskip

\noindent Since \eqref{eq:conic_problem} is a conic convex programming
problem, many special cases of \eqref{eq:conic_problem}
can be solved in polynomial time, at least in theory, using interior point
methods. However, in practice, the interior point methods are not able
to solve very large
instances of~\eqref{eq:conic_problem} because the computational
complexity of a matrix
factorization step, which is essential in these methods, becomes prohibitive.
On the other hand, the computational bottleneck in the \alcc\ is the
projection \eqref{eq:nonsmooth-operation}. In many optimization
problems that arise in applications, this projection can be solved
very efficiently as is the case with noisy compressed
sensing and matrix completion problems discussed in \cite{Ser10_1J}, and the convex
optimization problems with semidefinite constraints discussed above.
The convergence results above imply that the \alcc\ can solve very large instances of
\eqref{eq:conic_problem} very efficiently provided the corresponding
projection~\eqref{eq:nonsmooth-operation} can be solved
efficiently. The numerical results reported in~\cite{Aybat11_1J,Ser10_1J} for a
special case of \alcc\ provide evidence that our proposed algorithm
can be scaled to solve very large instances of the conic
problem~\eqref{eq:conic_problem}.

\subsection{Previous work}
\label{sec:previous}
Rockafellar~\cite{Roc73_1J} proposed an inexact augmented Lagrangian
method to solve problems of the form
\begin{equation}
\label{eq:rockafellar_problem}
p^*=\min \big\{p(x) : f(x)\geq 0,~x \in \chi\big\},
\end{equation}
where $\chi\subset\reals^n$ is a closed convex set,
$p:\reals^n\rightarrow\reals\cup\{+\infty\}$ is a convex function and
$f:\reals^n\rightarrow\reals^m$ such that each
component $f_i(x)$ of $f = (f_1, \ldots, f_m)$ is a
concave function for $i=1,\ldots,m$. 
Rockafellar~\cite{Roc73_1J} defined the ``penalty'' Lagrangian
\begin{equation}
\label{eq:Lagrangian_rock}
\tilde{\cL}_\mu(x,y):=p(x)+\frac{\mu}{2}~\Big\|\Big(\frac{y}{\mu}-f(x)\Big)_+\Big\|_2^2
-\frac{\norm{y}_2^2}{2\mu},
\end{equation}
where $(\cdot)_+:=\max\{\cdot,\mathbf{0}\}$ and $\max\{\cdot,\cdot\}$ are
componentwise operators, and $\mu$ is a fixed penalty parameter.
Rockafellar~\cite{Roc73_1J} established that given $y_0 \in \reals^m$,
the primal-dual iterates  sequences
$\{x_k, y_k\} \subset \chi\times \reals^m$ 
computed according to
\begin{align}
\tilde{\cL}_\mu(x_k,y_k)&\leq\inf_{x\in\chi}\tilde{\cL}_\mu(x,y_k)
+\alpha_k, \label{eq:x-update_rock}\\
y_{k+1}&=\left(y_k+\mu f(x_k)\right)_+, \label{eq:y-update_rock}
\end{align}
satisfy $\lim_{k\in\integers_+} p(x_k)=\bar{p}$ and $\limsup_{k\in\integers_+}
f(x_k)\leq \mathbf{0}$ when
\eqref{eq:rockafellar_problem} has a KKT point and  the parameter sequence $\{\alpha_k\}$
satisfies the summability condition
$\sum_{k=1}^{\infty}\sqrt{\mu~\alpha_k}<\infty$. Martinet~\cite{Mar78_1J} later
showed that the summability condition on
parameter sequence $\{\alpha_k\}$ is not necessary. However, in both
\cite{Mar78_1J,Roc73_1J} no iteration complexity result was given for
the algorithm \eqref{eq:x-update_rock}--\eqref{eq:y-update_rock} when $p$
was not continuously twice differentiable.

In this paper we show convergence rate results for an
augmented Lagrangian algorithm where we allow
penalty parameter $\mu$ to be a non-decreasing positive sequence $\{\mu_k\}$. After we
had independently
established these results, which are extensions of our previous results
in~\cite{Ser10_1J}, we became aware of a previous work by
Rockafellar~\cite{Rockafellar1976} where he proposed several different
variants of the algorithm in
\eqref{eq:x-update_rock}--\eqref{eq:y-update_rock} where $\mu$
could be updated between iterations. Rockafellar~\cite{Rockafellar1976}
established that for all non-decreasing positive multiplier sequences
$\{\mu_k\}$ satisfying the summability condition
$\sum_{k=1}^{\infty}\sqrt{\mu_k~\alpha_k}<\infty$, $\{y_k\}$ is bounded
and any limit point of $\{x_k\}$ is optimal to
\eqref{eq:rockafellar_problem}; moreover,
\begin{align}
  \label{eq:rock-rate}
  \max_{i=1,\ldots, m}\{f_i(x_k)\}\leq
  \frac{\norm{y_{k+1}-y_k}_2}{\mu_k},
  \quad p(x_k)-p^* \leq
  \frac{1}{2\mu_k}(\alpha_k+\norm{y_k}_2^2).
\end{align}
Note that the results in \cite{Rockafellar1976} only provide an
\emph{upper bound} on the sub-optimality; no lower bound is
provided. Since the iterates $\{x_k\}$ are only feasible in the limit,
it is possible that $p(x_k) \ll p^\ast$ and establishing a lower bound on the
sub-optimality is critical. Moreover,
Rockafellar~\cite{Rockafellar1976} does not discuss how to compute
iterates satisfying~\eqref{eq:x-update_rock} and assumes that a
black-box oracle produces such iterates; consequently, there are no
basic operation level complexity bounds in~\cite{Rockafellar1976}.

 In this paper, we extend 
  \eqref{eq:rockafellar_problem} 
  to a conic convex program where  $f(x)=Ax-b$, and $\cK$  is a closed,
  convex cone. We show that primal ALCC iterates
  $\{x_k\}\subset\chi$ satisfies $d_\cK(Ax_k-b)\leq\cO(\mu_k^{-1})$ and
  $|p(x_k)-p^*|\leq\cO(\mu_k^{-1})$, i.e. we provide \emph{both} an upper and a
  lower bound, using an inexact stopping condition that is an extension of
  \eqref{eq:x-update_rock}.  \alcc\ calls an optimal first order method, such as
  FISTA~\cite{Beck09_1J}, to compute an iterate
  $x_k$ satisfying a stopping condition similar to
  \eqref{eq:x-update_rock}. By
  carefully selecting the sub-optimality parameter sequence $\{\alpha_k\}$ and the
  penalty parameter sequence $\{\mu_k\}$, we are able to establish a bound
  on the 
  number of generalized
  projections of the form \eqref{eq:nonsmooth-operation} required  
  to obtain an $\epsilon$-feasible and $\epsilon$-optimal solution to
  \eqref{eq:conic_problem}, and also provide an operation level
  complexity bound.

In \cite{Rockafellar1976}, Rockafellar also provides an iteration complexity result for a different inexact augmented Lagrangian method. Given a non-increasing sequence $\{\alpha_k\}$ and a non-decreasing
  sequence $\{\mu_k\}$ such that  $\sum^{\infty}_{k=1}\sqrt{\mu_k~\alpha_k}<\infty$,
  the infeasiblity and suboptimality can be \emph{upper} bounded (see~\eqref{eq:rock-rate})
  when the duals $\{y_k\}$ are updated according to
  \eqref{eq:y-update_rock} and the primal iterates $\{x_k\}$
  satisfy
  \begin{align}
    \inf\{\norm{s}_2:\
    s\in\partial\phi_k(x_k)\}\leq\sqrt{\frac{\alpha_k}{\mu_k}},
    \label{eq:rock-dist}
  \end{align}
  where $\phi_k(x):=
  \tilde{\cL}_{\mu_k}(x,y_k)+\mathbf{1}_\chi(x)+\frac{1}{2\mu_k}\norm{x-x_{k-1}}_2^2$,
  $\tilde{\cL}_{\mu_k}$ is defined in \eqref{eq:Lagrangian_rock} and $\mathbf{1}_\chi$ is the indicator function of the closed convex set $\chi$.
  With this new stopping condition, Rockafellar~\cite{Rockafellar1976} was
  able to establish a \emph{lower} bound $p(x_k)-p^*\geq
  -\cO(\mu_k^{-1})$. Note that the stopping condition \eqref{eq:rock-dist}
  is much stronger than \eqref{eq:x-update_rock} -- in this paper we establish
  the lower bound using the weaker stopping condition \eqref{eq:x-update_rock}.

  First order methods for minimizing
  functions with Lipschitz continuous
  gradients~\cite{Nesterov04,Nesterov05} (and also the  non-smooth
  variants~\cite{Beck09_1J,Tseng08}) can only guarantee convergence in
  function values; therefore, the subgradient condition~(\ref{eq:rock-dist})
  has to be re-stated
  in terms of function values in order to use a first-order algorithm to
  compute the iterates. This is impossible when the objective function is
  non-smooth. Therefore, one cannot establish operational level
  complexity results for a method that uses the gradient stopping condition \eqref{eq:rock-dist}
  with first order methods. 
  Next, consider the case where $p$ is smooth, i.e. $\rho(\cdot)
  =0$. Suppose $\chi=\reals^n$,
  $\grad\gamma$ is Lipschitz continuous with constant $L_\gamma$ and
  $f(x)=Ax-b$. Then, it is easy to establish that $\grad \phi_k$ is also
  Lipschitz  continuous with Lipschitz constant
  $L_\phi=L_\gamma+\mu_k\sigma_{\max}^2(A)+\mu_k^{-1}=\cO(\mu_k)$.
  Since
  $\phi_k(x_k)-\inf_{x\in\reals^n}\phi_k(x)\leq\xi$ implies that
  $\norm{\grad \phi_k(x_k)}_2\leq\sqrt{2L_\phi\xi}$,  in order to
  ensure \eqref{eq:rock-dist} one has to set
  $\xi\leq\frac{1}{2\sigma_{\max}^2(A)}\frac{\alpha_k}{\mu_k^2}$. Thus,
  the complexity of computing each iterate $x_k$ satisfying \eqref{eq:rock-dist} will be significantly
  higher than the complexity of computing $x_k$ satisfying \eqref{eq:x-update_rock}, which is the one used in the \alcc.
  Therefore,
  although Rockafellar's method using \eqref{eq:rock-dist} has the same iteration complexity with
  \alcc, the operational level complexity of a first-order algorithm based on
  the gradient stopping criterion~\eqref{eq:rock-dist} will be
  significantly higher than the complexity of the \alcc\
  where $\xi=\alpha_k$. In summary, Rockafellar~\cite{Rockafellar1976} is
  only able to show an
  upper bound on sub-optimality of iterates for the stopping
  criterion~\eqref{eq:x-update_rock} that leads to an efficient algorithm;
  whereas the subgradient stopping criterion~\eqref{eq:rock-dist} that
  results in a lower bound is not practical for a first-order algorithm.

In \cite{Lan11_1J}, Lan, Lu and Monteiro 
consider problems of the form
\begin{equation}
  \label{eq:lan-et-al-11}
  \min\{\fprod{c,x}:~Ax=b,~x\in\cK\},
\end{equation}
where $\cK$ is a closed convex cone.
They proposed 
computing an approximate solution for \eqref{eq:lan-et-al-11} by minimizing the Euclidean
distance to the set of KKT points using Nesterov's
accelerated proximal gradient algorithm~(APG)
\cite{Nesterov04,Nesterov05}. They show that at most
$\cO\left(\epsilon^{-1}\right)$ iterations of Nesterov's APG
algorithm~\cite{Nesterov04,Nesterov05} suffice to compute a point
whose distance to the set of KKT points is at most $\epsilon>0$. In
\cite{Lan12_1J}, Lan and Monteiro proposed a first-order penalty
method to solve the following more general problem
\begin{equation}
\label{eq:conic_problem_lan}
\min\{\gamma(x):~Ax-b\in\cK,~x\in\chi\},
\end{equation}
where $\gamma$ is a convex function with Lipschitz continuous
gradient, $\cK$ is a closed, convex cone, $\chi$ is a simple convex
compact set and $A\in\reals^{m\times n}$. In order to solve
\eqref{eq:conic_problem_lan}, they used Nesterov's APG algorithm on
the
perturbed penalty problem
\begin{equation*}
\min\{\gamma(x)+\xi\norm{x-x_0}_2^2+\frac{\mu}{2}~d_{\cK}(Ax-b)^2:~x\in\chi\},
\end{equation*}
where $x_0\in\chi$, $d_\cK$ is as defined in \eqref{eq:dist_func}, and
$\xi>0$, $\mu>0$ are fixed perturbation and penalty parameters. They
showed that Nesterov's APG algorithm can compute a primal-dual
solution $(\tilde{x},\tilde{y})\in\chi\times\cK^*$ satisfying
$\epsilon$-perturbed KKT conditions
\begin{equation}
\label{eq:pertubed-KKT}
\fprod{\tilde{y},~\Pi_\cK(A\tilde{x}-b)}=0,\quad
d_\cK(A\tilde{x}-b)\leq\epsilon,\quad \grad
\gamma(\tilde{x})-A^T\tilde{y}\in-\cN_\chi(\tilde{x})+\cB(\epsilon),
\end{equation}
using $\cO\left(\epsilon^{-1}\log\left(\epsilon^{-1}\right)\right)$
projections onto $\cK$ and $\chi$, where
$\cN_\chi(\tilde{x}):=\{s\in\reals^n:~\fprod{s, x-\tilde{x}}\leq
0,~\forall x\in\chi\}$ and
$\cB(\epsilon):=\{x\in\reals^n:~\norm{x}_2\leq\epsilon\}$. Note that
since $\xi$ and $\mu$ are fixed, additional iterations of the
Nesterov's APG algorithm will not improve the quality of the
solution.

The optimization problem \eqref{eq:conic_problem_lan} is a special case of
\eqref{eq:conic_problem} with $\rho(\cdot)= 0$. 
Thus, ALCC can solve \eqref{eq:conic_problem_lan}. We show that every
limit point of the ALCC iterates are optimal for
\eqref{eq:conic_problem_lan}. Furthermore,  for \emph{any} $\epsilon >0$,
ALCC iterates are  $\epsilon$-optimal,
 and $\epsilon$-feasible for \eqref{eq:conic_problem_lan}
within $\cO\left(\epsilon^{-1}\log\left(\epsilon^{-1}\right)\right)$
projections onto $\cK$ and $\chi$ as
is the case with the algorithm proposed in \cite{Lan12_1J}.

Lan and Monteiro~\cite{Lan12_2J} proposed an inexact augmented
Lagrangian method to solve a special case of \eqref{eq:conic_problem}
with $\cK=\{\mathbf{0}\}$ and $\rho(\cdot)=0$; and showed that
Nesterov's APG algorithm can compute a primal-dual solution
$(\tilde{x},\tilde{y})\in\chi\times\reals^m$ satisfying
\eqref{eq:pertubed-KKT} using
$\cO\left(\epsilon^{-1}\left(\log\left(\epsilon^{-1}\right)\right)^{\frac{3}{4}}
  \log\log\left(\epsilon^{-1}\right)\right)$
projections onto $\chi$ and $\cK$.

Aybat and Iyengar~\cite{Ser10_1J}  proposed an inexact augmented
Lanrangian algorithm (FALC) to solve the
composite norm minimization problem
\begin{equation}
\label{eq:composite-norm}
\min_{X\in\reals^{m\times
    n}}\{\mu_1\norm{\sigma(\cF(X)-G)}_\alpha+\mu_2\norm{\cC(X)-d}_\beta+\gamma(X):\
\cA(X)-b\in\cQ\},
\end{equation}
where the function $\sigma(\cdot)$ returns the singular values of its
argument; $\alpha$ and $\beta\in\{1,2,\infty\}$; $\cA,\cC,\cF$ are
linear operators such that either $\cC$ or $\cF$ is injective, and
$\cA$ is surjective; $\gamma$ is a convex function with a Lipschitz
continuous gradient and $\cQ$ is a closed convex set. It was shown
that any limit point of the FALC iterates is an optimal solution of
the composite norm minimization problem~\eqref{eq:composite-norm}; and
for all $\epsilon>0$, the FALC iterates are $\epsilon$-feasible and
$\epsilon$-optimal after
$\cO\left(\log\left(\epsilon^{-1}\right)\right)$ FALC iterations,
which require $\cO\left(\epsilon^{-1}\right)$ shrinkage type
operations and Euclidean projection onto the set $\cQ$. The limitation
of FALC is that it requires $\cA$ to be a surjective mapping.
Consider a feasible set of the form
\begin{equation}
\label{eq:product-cone}
\{x\in\reals^n:~A_1x-b_1\in\cK_1,\ A_2x-b_2\in\cK_2,\ x\in\chi\},
\end{equation}
where $\cK_i$ is a closed convex cone, $A_i\in\reals^{m_i\times n}$ and
$b_i\in\reals^{m_i}$ for $i=1,2$. The set in \eqref{eq:product-cone} can
be reformulated as the feasible set in \eqref{eq:conic_problem} by
choosing
$A=\left(
\begin{array}{c}
    A_1 \\
    A_2 \\
\end{array}
\right)$ and $\cK=\cK_1\times\cK_2$, where $m=m_1+m_2$. FALC can work with such a set
only if $A$ has linearly independent rows, i.e., $\Rank(A)=m_1+m_2$. This
is a severe limitation for the
practical problem. On the other hand, the \alcc~works for the feasible sets
of the form \eqref{eq:product-cone} without any additional
assumption. Thus, ALCC can be used to solve much larger class of
optimization problems.

In our opinion the \alcc~proposed in this paper unifies all the previous work on fast first-order penalty and/or augmented Lagrangian algorithms for solving optimization problems that are special cases of \eqref{eq:conic_problem}. We do not impose any regularity conditions on the constraint matrix $A$ and the projection step \eqref{eq:nonsmooth-operation} is the natural extension of the gradient projection step. We believe that this unified treatment will spur further research in understanding the limits of performance of the first order algorithms for general conic problems.
\section{Preliminaries}
\label{sec:preliminaries}
In Section~\ref{sec:apg}, first we briefly discuss a variant of Nesterov's APG
algorithm~\cite{Nesterov04,Nesterov05} to solve
\eqref{eq:conic_problem} without conic constraints. Next, we introduce
a dual function for the conic problem in \eqref{eq:conic_problem} and
establish some of its properties in Section~\ref{sec:prelims}. The definitions and the results of Section~\ref{sec:prelims}
are extensions of the corresponding definitions and results 
in \cite{Rockafellar-73,Roc73_1J}, 
to the case where $\cK\subset\reals^m$ is a general closed, convex cone.
\subsection{Accelerated Proximal Gradient~(APG) algorithm}
\label{sec:apg}
\begin{figure}[ht]
    \rule[0in]{6.5in}{1pt}\\
    \textbf{Algorithm APG}$(\bar{\rho}, \bar{\gamma}, \chi, x_0,\textsc{stop})$\\
    \rule[0.125in]{6.5in}{0.1mm}
    \vspace{-0.25in}
    {\small
    \begin{algorithmic}[1]
    \STATE $x_0^{(1)}\gets x_0$, $x_1^{(2)} \gets x_0$, $t_1\gets 1$, $\ell\gets 0$
    \WHILE{\textsc{stop} is \FALSE} \label{algeq:stop}
    \STATE $\ell \gets \ell + 1$
    \STATE $x_\ell^{(1)} \gets \argmin\left\{\bar{\rho}(x)+\fprod{\nabla
        \bar{\gamma}\left(x_\ell^{(2)}\right),~x-x_\ell^{(2)}} +
      \frac{L_{\bar{\gamma}}}{2}~\norm{x-x_\ell^{(2)}}_2^2: x\in
      \chi\right\}$ \label{algeq:tseng_z}
    \STATE $t_{\ell+1}\gets \left(1+\sqrt{1+4~t^2_{\ell}}\right)/2$
    \STATE $x_{\ell+1}^{(2)} \gets x_\ell^{(1)} +
    \left(\frac{t_{\ell}-1}{t_{\ell+1}}\right)\left(x_\ell^{(1)}-x_{\ell-1}^{(1)}\right)$
    \ENDWHILE
    \end{algorithmic}
    \rule[0.25in]{6.5in}{0.1mm}
    }
    \vspace{-0.5in}
    \caption{Accelerated Proximal Gradient Algorithm}\label{alg:apg}
    \vspace{-0.25in}
\end{figure}
In this section we state and briefly discuss the details of a particular
implementation of Fast Iterative Shrinkage-Thresholding
Algorithm~\cite{Beck09_1J} (FISTA), which extends Nesterov's
accelerated proximal gradient algorithm~\cite{Nesterov04,Nesterov05}
for minimizing smooth convex functions over simple convex sets, to
solve non-smooth convex minimization problems.

FISTA computes an $\epsilon$-optimal solution to
$\min\{\bar{\rho}(x)+\bar{\gamma}(x):\ x\in\reals^n\}$
in $\cO\left(\epsilon^{-\frac{1}{2}}\right)$ iterations, where
$\bar{\rho}:\reals^n \rightarrow\reals$ and $\bar{\gamma}:\reals^n
\rightarrow\reals$ are continuous convex functions such that $\grad
\bar{\gamma}$ is Lipschitz continuous on $\reals^n$ with constant
$L_{\bar{\gamma}}$.  Tseng~\cite{Tseng08} showed that this rate
result for FISTA also holds when $\bar{\rho}:\reals^n \rightarrow
(-\infty, +\infty]$ and $\bar{\gamma}:\reals^n \rightarrow
(-\infty, +\infty]$ are proper, lower semicontinuous, and convex
functions such that $\dom \bar{\rho}$ is closed and $\grad
\bar{\gamma}$ is Lipschitz
continuous on $\reals^n$.

This extended version of FISTA is displayed in Figure~\ref{alg:apg} as \apg. Hence, FISTA can solve constrained
problems of the form
\begin{align}
\label{eq:tseng_problem}
\min\{\bar{\rho}(x)+\bar{\gamma}(x):\ x\in\chi\},
\end{align}
where $\chi\subset\reals^n$ is a simple closed convex set.

The \apg~displayed in Figure~\ref{alg:apg} takes as
input the functions $\bar{\rho}$ and $\bar{\gamma}$, the simple closed
convex set $\chi\subset\reals^n$, an initial iterate
$x^{(0)}\in\chi$ and a stopping criterion \textsc{stop}.
Lemma~\ref{lem:tseng_corollary} gives the iteration complexity of
the \apg.
\begin{lemma}
\label{lem:tseng_corollary}
Let $\bar{\rho}$ and $\bar{\gamma}$ be a proper, closed,
convex functions
such that $\dom \bar{\rho}$ is closed and $\grad \bar{\gamma}$ is
Lipschitz continuous on
$\reals^n$ with constant $L_{\bar{\gamma}}$. Fix $\epsilon>0$ and let
$\{x_\ell^{(1)},x_\ell^{(2)}\}$ denote the sequence of iterates
computed by the \apg\ when $\textsc{stop}$
is disabled. Then
$\bar{\rho}\big(x_\ell^{(1)}\big)+\bar{\gamma}\big(x_\ell^{(1)}\big)\leq
\min\{\bar{\rho}(x)+\bar{\gamma}(x):\ x\in\chi\}+\epsilon$ whenever
$\ell\geq\sqrt{\frac{2L_{\bar{\gamma}}}{\epsilon}}~\norm{x^\ast-x_0}_2-1$,
where $x^\ast
\in \argmin\{\bar{\rho}(x) + \bar{\gamma}(x):\ x\in\chi\}$.
\end{lemma}
\begin{proof}
See Corollary~3 in \cite{Tseng08} and Theorem 4.4 in \cite{Beck09_1J} for
the details of proof.\end{proof}

\subsection{A dual function for conic convex programs and its properties}
\label{sec:prelims}
For all $\mu\geq 0$, optimization problem $(P)$ in
\eqref{eq:conic_problem} is equivalent to
\begin{equation}
\label{eq:penalty_problem}
\min \left\{ p(x)+\frac{\mu}{2}\norm{Ax-s-b}_2^2 :  Ax-s=b,~x \in \chi,~s
  \in \cK\right\}.
\end{equation}
Let $y\in\reals^m$ denote a Lagrangian dual variable corresponding to
the equality constraint in \eqref{eq:penalty_problem}, and let 
\begin{equation}
\label{eq:Lmu_def}
\cL_\mu(x,y):=\min_{s\in\cK} \left\{ p(x)
-\fprod{y, Ax-s-b}+\frac{\mu}{2}\norm{Ax-s-b}_2^2\right\}
\end{equation}
denote the ``penalty'' Lagrangian function for
\eqref{eq:penalty_problem} with $\dom \cL_\mu=\chi \times \reals^m$. 
For $\mu>0$,
\begin{align}
\cL_\mu(x,y)&=p(x)
+\frac{\mu}{2}~\left ( \min_{s\in\cK}
  \left\|Ax-s-b-\frac{y}{\mu}\right\|_2^2-\frac{\norm{y}_2^2}{\mu^2}
\right),
\nonumber\\
&=p(x)
+\frac{\mu}{2}~d_\cK\left(Ax-b-\frac{y}{\mu}\right)^2
-\frac{\norm{y}_2^2}{2\mu}, \label{eq:L_mu}
\end{align}
where $d_{\cK}(\cdot)$ is the distance function defined
in~\eqref{eq:dist_func}.
When $\mu = 0$,  the definition in \eqref{eq:Lmu_def} implies that
\begin{equation}
  \label{eq:lagrangian_0}
  \cL_0(x,y)=\left\{
    \begin{array}{ll}
      p(x)-\fprod{y, Ax-b}, & y\in\cK^*, \\
      -\infty, & \mbox{otherwise.}
    \end{array}
  \right.
\end{equation}
For $\mu \geq 0$, we define a dual function $g_\mu:\reals^m\rightarrow\reals$ for \eqref{eq:conic_problem} such that
\begin{equation}
\label{eq:g_mu}
g_\mu(y):=\inf_{x\in\chi}\cL_\mu(x, y).
\end{equation}
Note that from \eqref{eq:lagrangian_0} it follows that $g_0$ is the
Lagrangian dual function of $(P)$.

The definitions above and the results detailed below are immediate extensions of
corresponding definitions and results in~\cite{Rockafellar-73}, given for $\cK=\reals^m_+$, to the case where $\cK$ is a general closed convex cone.
We state and prove the extensions here for the sake of completeness.
These results are used in Section~\ref{sec:alcc} to establish the convergence properties of ALCC iterate sequence.
\begin{lemma}
\label{lem:L-cvx}
For all $\mu \geq 0$, $x\in\chi$ and $y\in\reals^m$, $\cL_\mu$ defined in \eqref{eq:Lmu_def} satisfies
\begin{align}
\label{eq:LmuF}
\cL_\mu(x,y)=\inf_{u\in\reals^m} \left\{F_\mu(x,u)-\fprod{y,u}\right\},
\end{align}
where $F_\mu:\chi\times \reals^m\rightarrow\reals\cup\{+\infty\}$ is defined as follows
\begin{equation}
\label{eq:Fmu}
F_\mu(x,u):=\left\{
              \begin{array}{ll}
                p(x)+\frac{\mu}{2}\norm{u}_2^2, & \hbox{if $Ax-b\in\cK+u$}, \\
                +\infty, & \hbox{otherwise.}
              \end{array}
            \right.
\end{equation}
Hence, $\cL_\mu(x,y)$ is convex in $x\in\chi$ and concave in
$y\in\reals^m$, and $g_\mu(y)$ defined in
\eqref{eq:g_mu} is concave in $y\in\reals^m$.
\end{lemma}
\begin{proof}
The representation in \eqref{eq:LmuF} trivially follows from the definition of $F_\mu$
in \eqref{eq:Fmu}. For a fixed $x\in\chi$, \eqref{eq:Lmu_def} implies that $\cL_\mu(x,y)$ is the
infimum  of  affine functions of $y$, hence $\cL_\mu(x, y)$ is concave in $y$. Hence, $g_\mu$ defined in
\eqref{eq:g_mu} is the infimum of concave functions; therefore, it is also concave. For a fixed $y\in\reals^m$, when $\mu>0$, convexity of $\cL_\mu(x,y)$ in $x$ follows from \eqref{eq:L_mu} and the fact that $p(\cdot)$ and $d_\cK(\cdot)$ are convex functions; otherwise, when $\mu=0$, it trivially follows from \eqref{eq:lagrangian_0}.
\end{proof}
\begin{lemma}
\label{lem:prox}
Let $g:\reals^m\rightarrow\reals\cup\{+\infty\}$ be a proper closed convex
function. For $\mu > 0$, let
\eq
\psi_\mu(y)=\min_{z\in\reals^m}\Big\{g(z)+\frac{1}{2\mu}\norm{z-y}_2^2\Big\}, \quad \pi_\mu(y) = \argmin_{z\in\reals^m}\Big\{g(z)+\frac{1}{2\mu}\norm{z-y}_2^2\Big\}
\en
denote the Moreau regularization of and the proximal map corresponding to $g$, respectively. Then, for all $y_1,y_2\in\reals^m$,
\begin{equation}
\label{eq:prox}
\norm{\pi_\mu(y_1)-\pi_\mu(y_2)}_2^2+\norm{\pi^c_\mu(y_1)-\pi^c_\mu(y_2)}_2^2\leq\norm{y_1-y_2}_2^2,
\end{equation}
where $\pi^c_\mu(y) := y-\pi_\mu(y)$ for all $z\in\reals^m$. Moreover,
$\psi_\mu:\reals^m\rightarrow\reals$ is an everywhere finite, differentiable convex function such that
\begin{equation}
\label{eq:moreau-grad}
\grad\psi_\mu(y)=\frac{1}{\mu}~(y-\pi_\mu(y))=\frac{1}{\mu}~\pi^c_\mu(y),
\end{equation}
is Lipschitz continuous with constant $\frac{1}{\mu}$.
\end{lemma}
\begin{proof}
The proof of \eqref{eq:prox} is given in \cite{Rockafellar-76} and the rest of the claims including \eqref{eq:moreau-grad} are shown in \cite{Hiriart-Urruty-Lemarechal-1993}.
\end{proof}
\begin{theorem}
\label{thm:moreau}
Suppose Assumption~\ref{asp:KKT} holds.
Then, for any $\mu>0$, $g_\mu$ is an everywhere finite, continuously differentiable concave function and $g_\mu$ achieves its maximum value at any KKT
point. Moreover,
\begin{equation}
\label{eq:g_moreau}
g_\mu(y)=\max_{z\in\reals^m}
\left\{g_0(z)-\frac{1}{2\mu}\norm{z-y}_2^2 \right\},
\end{equation}
and
\begin{equation}
\label{eq:grad_gk}
\grad g_\mu(y)=-\frac{1}{\mu}(y-\pi_\mu(y)),
\end{equation}
is Lipschitz continuous with Lipschitz constant equal to $\frac{1}{\mu}$, where $\pi_\mu(y)\in\cK^*$ denotes the unique maximizer in \eqref{eq:g_moreau}.
\end{theorem}
\begin{proof}
Fix $\mu \geq 0$, define 
\begin{equation}
\label{eq:h_mu}
h_\mu(u):=\inf_{x\in\chi}F_\mu(x,u).
\end{equation}
Note that $F_{\mu}(x,u) =  p(x)+\frac{\mu}{2}\norm{u}_2^2 + \ones_{\cK}(Ax-b-u)$, where $\ones_{\cK}(\cdot)$ denotes the indicator function of the set $\cK$; therefore, $F_{\mu}(x,u)$ is convex in $(x,u)$. Since $F_\mu$ is convex in $(x,u)$, $\chi$ is a convex set and $h_{\mu}(0) = \inf_{x \in \chi} \{ p(x)+
\ones_{\cK}(Ax-b)\} = p^\ast > -\infty$, it follows that $h_{\mu}$ is a convex
function such that $h_\mu(\cdot)>-\infty$ \cite{Boyd04_1B}. From the definition of $F_{\mu}$, it
follows that for all $u\in\reals^m$,
\begin{equation*}
h_\mu(u)= h_0(u)+ \mu~\omega(u),
\end{equation*}
where $\omega(u) := \frac{1}{2}\norm{u}_2^2$. Substituting \eqref{eq:LmuF}
in \eqref{eq:g_mu}, for all $\mu \geq 0$, we get
\begin{equation*}
g_\mu(y)=\inf_{u\in\reals^m} \left\{h_\mu(u)-\fprod{y,u} \right\}=-h_\mu^*(y),
\end{equation*}
where $h_\mu^*$ denotes the conjugate of the convex function $h_\mu$.

Fix $\mu>0$, since $h_\mu$ is a sum of two convex functions, it follows from Theorem~16.4 in \cite{Rockafellar-book-70} that
\begin{equation}
\label{eq:conjugate_sum}
g_\mu(y)=-(h_0 + \mu \omega)^*(y)
 =-\min_{z\in\reals^m} \left\{h^*_0(z)+\mu~\omega^\ast\left(\frac{y-z}{\mu}\right)\right\}.
\end{equation}
Since $h^*_0=-g_0$ and $\omega^*=\omega$, the  result
\eqref{eq:g_moreau} immediately follows from \eqref{eq:conjugate_sum}.

Note that \eqref{eq:g_moreau} shows that $-g_\mu$ is the Moreau regularization of $-g_0$. Therefore, Lemma~\ref{lem:prox} and \eqref{eq:g_moreau} imply that $g_\mu$ is everywhere finite, differentiable concave function such that $\grad g_\mu$ is given in \eqref{eq:grad_gk}.

Let $y^*$ be a KKT point of \eqref{eq:conic_problem}. Note that $\pi_\mu(y^*)=y^*$. Hence $\grad g_\mu(y^*)=\mathbf{0}$. Concavity of $g_\mu$ implies that $y^*\in\argmax g_\mu(y)$ for any KKT point $y^*$.
\end{proof}
\begin{theorem}
\label{thm:grad_diff}
Fix $\mu>0$ and $\bar{y}\in\reals^m$. Suppose $\bar{x}\in\chi$ is an
$\xi$-optimal solution to $\min_{x\in\chi}L_{\mu}(x,\bar{y})$, i.e.
$
L_{\mu}(\bar{x},\bar{y}) \leq \min\{L_{\mu}(x,\bar{y}):\ x\in\chi\}
+ \xi = g_{\mu}(\bar{y}) + \xi.
$
Then
\begin{equation}
\label{eq:grad_diff}
\mu~\norm{\grad_y \cL_\mu(\bar{x},\bar{y})-\grad g_\mu(\bar{y})}_2^2\leq 2\xi.
\end{equation}
\end{theorem}
\begin{proof}
For $\mu>0$, $g_\mu$ is concave and $\grad g_\mu$ is Lipschitz continuous with Lipschitz constant equal to $\frac{1}{\mu}$; therefore,
\begin{equation}
\label{eq:lipschitz-g}
g_\mu(y)\geq g_\mu(\bar{y})+\fprod{\grad g_\mu(\bar{y}),
  y-\bar{y}}-\frac{1}{2\mu}\norm{y-\bar{y}}_2^2,
\end{equation}
for all $y\in\reals^m$. Moreover, since for every $x\in\chi$,
$\cL_\mu(x,y)$ is concave in $y$,  it follows that for all $y\in\reals^m$
\begin{equation}
\label{eq:lowerboundL}
\cL_\mu(\bar{x}, \bar{y})+\fprod{\grad_y
  \cL_\mu(\bar{x},\bar{y}),~y-\bar{y}}\geq \cL_\mu(\bar{x},y)\geq
g_\mu(y).
\end{equation}
Combining \eqref{eq:lipschitz-g}, \eqref{eq:lowerboundL} and the fact
that $\bar{x}$ is $\xi$-optimal and $y$ is arbitrary, we get
\eq
\xi\geq
\sup_{y\in\reals^m}\left\{\fprod{\grad
  g_\mu(\bar{y})-\grad_y \cL_\mu(\bar{x},\bar{y}),~
  y-\bar{y}}-\frac{1}{2\mu}~\norm{y-\bar{y}}_2^2\right\} =
\frac{\mu}{2} \norm{\grad
  g_\mu(\bar{y})-\grad_y \cL_\mu(\bar{x},\bar{y})}_2^2.
\en
\end{proof}
\section{ALCC Algorithm}
\label{sec:alcc}
In order to solve $(P)$ given in \eqref{eq:conic_problem}, we
inexactly solve the sequence of  sub-problems: 
\begin{equation}
\label{eq:subproblem}
(SP_k):\ \min_{x\in\chi} P_k(x,y_k),
\end{equation}
where
\eq
P_k(x,y) :=\frac{1}{\mu_k}~\cL_{\mu_k}(x,y) = \frac{1}{\mu_k}~p(x) +
\frac{1}{2}~d_\cK\left(Ax-b-\frac{y}{\mu_k}\right)^2.
\en
For notational convenience, we define
\eq
f_k(x,y) :=\frac{1}{2}~d_\cK\left(Ax-b-\frac{y}{\mu_k}\right)^2.
\en
Therefore, $P_k(x,y) = \frac{1}{\mu_k}~p(x) + f_k(x,y)$.
The specific choice of penalty parameter and Lagrangian dual sequences, $\{\mu_k\}$ and $\{y_k\}$, are discussed later in this section.
\begin{lemma}
\label{lem:lipschitz}
For all $k\geq 1$ and $y\in\reals^m$, $f_k(x,y)$ is convex in $x$. Moreover,
\begin{equation}
\label{eq:gradfk}
\grad_x f_k(x,y)=A^T\left(Ax-b-\frac{y}{\mu_k}-\Pi_\cK\left(Ax-b-\frac{y}{\mu_k}\right)\right),
\end{equation}
and $\grad_x f_k(x,y)$ is Lipschitz continuous in $x$ with constant $L=\sigma^2_{\max}(A)$.
\end{lemma}
\begin{proof}
See appendix for the proof.
\end{proof}
\begin{figure}[t]
    \rule[0in]{6.5in}{1pt}\\
    \textbf{Algorithm ALCC}~$(x_0,~\{\alpha_k,~\eta_k,~\mu_k\})$\\
    \rule[0.125in]{6.5in}{0.1mm}
    \vspace{-0.25in}
    {\small
    \begin{algorithmic}[1]
    \STATE $y_1\gets \mathbf{0}$, $k\gets 1$
    \WHILE{$k\geq 1$}
    \STATE $x_k\gets\oracle(P_k, y_k, \alpha_k, \eta_k, \mu_k)$ /*~\emph{See Section~\ref{sec:oracle} for \oracle} ~*/
    \STATE $y_{k+1}\gets\mu_k\left[\Pi_\cK\left(Ax_k-b-\frac{y_k}{\mu_k}\right)-\left(Ax_k-b-\frac{y_k}{\mu_k}\right)\right]$ \label{algeq:y_update}
    \STATE $k\gets k+1$
    \ENDWHILE
    \end{algorithmic}
    \rule[0.25in]{6.5in}{0.1mm}
    }
    \vspace{-0.5in}
    \caption{Augmented Lagrangian Algorithm for Conic Convex Programming} \label{alg:alcc}
    \vspace{-0.25in}
\end{figure}

The \alcc\ is displayed in Figure~\ref{alg:alcc}. The inputs to ALCC are
an initial point $x_0\in\chi$
and a parameter sequence
$\{\alpha_k,~\eta_k,~\mu_k\}$ such that
\begin{equation}
  \label{eq:params}
  \alpha_k\searrow 0, \quad \eta_k\searrow 0, \quad 0<\mu_k\nearrow \infty.
\end{equation}
\subsection{Oracle}
\label{sec:oracle}
The subroutine $\oracle (P, \bar{y}, \alpha, \eta, \mu)$ returns
$\bar{x} \in \chi$ such that $\bar{x}$ satisfies one of the following
two conditions:
\begin{align}
&0\leq P(\bar{x},\bar{y})-\inf_{x\in\chi}P(x,\bar{y})\leq \frac{\alpha}{\mu}, \label{eq:eps_opt}\\
&\exists q\in\partial_x P(\bar{x},\bar{y})+\partial_x \mathbf{1}_\chi(\bar{x})~  \mbox{ s.t. }
  \norm{q}_2\leq\frac{\eta}{\mu}, \label{eq:grad_opt}
\end{align}
where $\mathbf{1}_\chi(\cdot)$ denotes the indicator function of the set $\chi$.

Let $\bar{\rho}_k(x):=\frac{1}{\mu_k}~\rho(x)$ and
$\bar{\gamma}_k(x):=\frac{1}{\mu_k}~\gamma(x)+f_k(x,y_k)$. Then
$\grad \bar{\gamma}_k$ exists and is Lipschitz continuous with
Lipschitz constant
\begin{equation}
\label{eq:lipschitz-constant}
L_{\bar{\gamma}_k}:= \frac{1}{\mu_k}~L_\gamma+\sigma^2_{\max}(A).
\end{equation}
Let
\begin{equation}
\label{eq:chi-opt-k}
\chi\supset\chi_k^*:=\argmin_{x\in \chi}P_k(x,y_k)
\end{equation}
denote the set of optimal solutions to $(SP_k)$. Then, Lemma~\ref{lem:tseng_corollary}
guarantees that the \apg~with the initial iterate $x_{k-1}\in\chi$
requires at most
\begin{equation}
\label{eq:lmax}
\ell_{\max}(k):=\sqrt{\frac{2\mu_k L_{\bar{\gamma}_k}}{\alpha_k}}~d_{\chi_k^*}(x_{k-1})
\end{equation}
iterations to compute
$\frac{\alpha_k}{\mu_k}$-optimal solution to the $k$-th subproblem
$(SP_k)$ in \eqref{eq:subproblem}. Thus, setting the stopping
criterion $\textsc{stop} = \{l
\geq \ell_{\max}(k)\}$ ensures that the output of the \apg\ satisfies
\eqref{eq:eps_opt}. Thus, we have
shown that  
there exists a subroutine 
$\oracle(P_k, y_k, \alpha_k, \eta_k, \mu_k)$ that can compute
$x_k$ satisfying either \eqref{eq:eps_opt} or \eqref{eq:grad_opt}. 
As indicated earlier,
the computational complexity of  each
iteration in the \apg\ is 
dominated by the complexity of computing the solution to~\eqref{eq:nonsmooth-operation}.
\subsection{Convergence properties of \alcc}
In this section we investigate the convergence rate of \alcc.
\begin{lemma}
\label{lem:cone}
Let $\cK\subset\reals^n$ denote a closed, convex cone and
$\bar{x}\in\reals^n$. Then $\bar{x}-\Pi_\cK(\bar{x})\in-\cK^*$ and
$\fprod{\bar{x}-\Pi_\cK(\bar{x}),~\Pi_\cK(\bar{x})}=0$, where
$\cK^*=\{s\in\reals^n:\ \fprod{s, x}\geq 0\; \forall
x\in\cK\}$. Finally, if $x\in-\cK^*$, then $\Pi_\cK(x)=\mathbf{0}$.
\end{lemma}
\begin{proof}
See appendix for the proof.
\end{proof}

From Lemma~\ref{lem:cone}, it follows that the dual variable $y_{k+1}$ computed in
Line~\ref{algeq:y_update} of \alcc~ satisfies $y_{k+1}\in\cK^*$. Also note
that  for all $k\geq 1$, 
\begin{equation}
\label{eq:ykp}
y_{k+1}=y_k+\mu_k\grad_y \cL_{\mu_k}(x_k, y_k).
\end{equation}

Next, we establish that the sequence of dual variables  $\{y_k\}$ generated by \alcc~is bounded for an appropriately chosen parameter
sequence.
\begin{lemma}
\label{lem:inexact-opt}
Let $\{x_k,y_k\}\in\chi\times\cK^*$ be the sequence of primal-dual
ALCC iterates for a given input parameter sequence
$\{\alpha_k,~\eta_k,~\mu_k\}$ satisfying \eqref{eq:params}. Then, for all
$k\geq 1$, 
\begin{equation}
0\leq \cL_{\mu_k}(x_k,y_k)-g_{\mu_k}(y_k)\leq \xik, \label{eq:eps_opt_L}
\end{equation}
where
\begin{equation}
\label{eq:xik}
\xik  = \max\{\alpha_k,~\eta_k~d_{\chi_k^*}(x_{k})\},
\end{equation}
and $\chi_k^*\subset\chi$ is defined in \eqref{eq:chi-opt-k}.
\end{lemma}
\begin{proof}
Fix $k\geq 1$. 
Suppose $x_k =\oracle(P_k, y_k, \alpha_k, \eta_k, \mu_k)$ satisfies
\eqref{eq:eps_opt}. Then we
have
\begin{equation}
\label{eq:inexact-subopt}
P_k(x_k, y_k)\leq \inf_{x\in\chi}P_k(x, y_k)+ \frac{\alpha_k}{\mu_k}=
\frac{g_{\mu_k}(y_k)+\alpha_k}{\mu_k}.
\end{equation}
Suppose instead that $x_k = \oracle(P_k, y_k, \alpha_k, \eta_k, \mu_k)$
satisfies \eqref{eq:grad_opt}. Then, there exists
$q_k\in\partial_x P_k(x_k,y_k)+\partial \mathbf{1}_{\chi}(x_k)$ such
that $\norm{q_k}_2\leq\frac{\eta_k}{\mu_k}$. Since
$P_k(x,y_k)+\mathbf{1}_{\chi}(x)$ is convex in $x$, it follows that
\begin{equation}
\label{eq:inexact-grad}
P_k(x_k,y_k)\leq \inf_{\bar{x}\in\chi_k^*}P_k(\bar{x},y_k)+\fprod{q_k, x_k-\bar{x}}\leq
\frac{g_{\mu_k}(y_k)+\eta_k~d_{\chi_k^*}(x_{k})}{\mu_k}.
\end{equation}
Since $P_k(x,y)=\frac{1}{\mu_k}\cL_{\mu_k}(x,y)$, the desired result
follows from \eqref{eq:inexact-subopt} and \eqref{eq:inexact-grad}.
\end{proof}\\
The following result was originally established in \cite{Roc73_1J} for $\cK = \reals^m_+$. We state and prove the
  extension to general convex cones for completeness.
\begin{theorem}
\label{thm:bounded-y}
Suppose $B:=\sum_{k=1}^{\infty}\sqrt{2~\xik\mu_k}<\infty$,
where $\xik$ is defined in \eqref{eq:xik}. Then, for all $k \geq 1$,  $\norm{y_k}_2\leq
B+\norm{y^*}_2$ where $y^*$ is  any  KKT
point of $(P) $. 
\end{theorem}
\begin{proof}
Lemma~\ref{lem:inexact-opt} and
Theorem~\ref{thm:grad_diff} imply that $\sqrt{2~\xik\mu_k} \geq\norm{\mu_k\grad_y\cL_{\mu_k}(x_k,y_k)-\mu_k\grad g_{\mu_k}(y_k)}_2.$
Next, adding and subtracting $y_k$, and using \eqref{eq:grad_gk} and
\eqref{eq:ykp}, we get
\begin{equation}
\sqrt{2~\xik\mu_k}  \geq  \norm{\mu_k\grad_y\cL_{\mu_k}(x_k,y_k) + y_k - (y_k
   +\mu_k\grad g_{\mu_k}(y_k))}_2 =
 \norm{y_{k+1}-\pi_{\mu_k}(y_k)}_2, \label{eq:y-My}
\end{equation}
Since $\sum_{k=1}^{\infty}\sqrt{2~\xik\mu_k}<\infty$, it follows that $\xik\mu_k\rightarrow 0$. Thus,
$\lim_{k\in\integers_+} \big(y_{k+1}-\pi_{\mu_k}(y_k)\big)=0$.

Assumption~\ref{asp:KKT} guarantees that a KKT point  $y^*\in\cK^*$ 
exists. Since
$y^*\in\argmax_{y\in\reals^m} g_0(y)$,  
Theorem~\ref{thm:moreau} implies that
$y^*\in\argmax_{y\in\reals^m} g_{\mu_k}(y)$ for all $k \geq 1$. Therefore, $\grad
g_{\mu_k}(y^*)=0$, and consequently, by  \eqref{eq:grad_gk},
$y^*=\pi_{\mu_k}(y^*)$.  Since $\pi_{\mu_k}$ is non-expansive, it follows
that
\begin{equation*}
\norm{\pi_{\mu_k}(y_k)-y^*}_2=\norm{\pi_{\mu_k}(y_k)-\pi_{\mu_k}(y^*)}_2\leq\norm{y_k-y^*}_2.
\end{equation*}
Hence,
\begin{eqnarray}
\norm{y_{k+1}-y^*}_2&\leq
&\norm{y_{k+1}-\pi_{\mu_k}(y_k)}_2+\norm{\pi_{\mu_k}(y_k)-y^*}_2,
\nonumber\\
&\leq &\norm{y_{k+1}-\pi_{\mu_k}(y_k)}_2+\norm{y_k-y^*}_2,\nonumber\\
&\leq & \sqrt{2~\xik\mu_k}+\norm{y_k-y^*}_2. \label{eq:y-induction-eps}
\end{eqnarray}
Since $y_1=\mathbf{0}$, the desired result is obtained by summing the
above inequality over $k$.
\end{proof}

In the rest of this section we investigate the convergence properties
of ALCC for the multiplier sequence $\{\alpha_k,\eta_k,\mu_k\}$ defined as follows
\begin{equation}
\label{eq:mult-seq}
  \mu_k = \beta^k~\mu_0,\quad \alpha_k = \frac{1}{k^{2(1+c)}~\beta^k}~\alpha_0,\quad \eta_k = \frac{1}{k^{2(1+c)}~\beta^k}~\eta_0,
\end{equation}
for all $k\geq 1$, where $\beta>1$, $c, \alpha_0, \eta_0$ and
$\mu_0$ are all strictly positive.  Thus, $\alpha_k\searrow0$, $\eta_k\searrow0$ and $\mu_k\nearrow\infty$.

Let $\infty>\Delta_\chi:=\max_{x\in\chi}\max_{x'\in\chi}\norm{x-x'}_2$ denote the diameter of the compact set $\chi$. Clearly, $d_{\chi_k^*}(x_{k})\leq\Delta_\chi$ for all $k\geq 1$, where $\chi_k^*\subset\chi$ is defined in \eqref{eq:chi-opt-k}.
Hence, from the definition of $\xi_k$ in \eqref{eq:xik}, it follows that
\begin{equation}
\sqrt{\xik\mu_k}\leq
\frac{1}{k^{1+c}}~\sqrt{\mu_0\max\{\alpha_0,~\eta_0\Delta_\chi\}},\quad
\forall k\geq 1,
\end{equation}
and $\sum_{k=1}^{\infty}\sqrt{\xik\mu_k}<\infty$ as
required by Theorem~\ref{thm:bounded-y}.  First, we  lower bound the
sub-optimality as a function of primal infeasibility of the iterates.
\begin{theorem}
\label{thm:subopt-lower}
Let $\{x_k,y_k\}\in\chi\times\cK^*$ be the sequence of primal-dual
ALCC  iterates corresponding to a parameter sequence
$\{\alpha_k, \eta_k, \mu_k\}$ satisfying \eqref{eq:params}. Then
\begin{equation*}
p(x_k)-p^*\geq-\norm{y^*}_2~d_\cK
\left(Ax_k-b-\frac{y_k}{\mu_k}\right)+\frac{1}{\mu_k}\fprod{y_k,y^*},
\end{equation*}
where $y^*\in\cK^*$ denotes any KKT point of $(P)$ and $p^*$ denotes the
optimal value of $(P)$ given in \eqref{eq:conic_problem}.
\end{theorem}
\begin{proof}
The dual function $g_0(y)=-\infty$ when $y\not\in\cK^*$; and for all
$y\in\cK^*$, the dual function $g_0$ of $(P)$ can be equivalently
written as
\begin{eqnarray*}
g_0(y)&= &\fprod{b,y}+\inf_{x\in\reals^n} \left\{p(x)+\mathbf{1}_{\chi}(x)-\fprod{A^Ty,x}\right\},\\
&=& \fprod{b,y}-(p+\mathbf{1}_\chi)^*(A^Ty).
\end{eqnarray*}
Hence, the dual of $(P)$ 
is
\begin{equation}
\label{eq:dualproblem}
(D): \quad \max_{y\in\cK^*}\fprod{b,y}-(p+\mathbf{1}_\chi)^*(A^Ty).
\end{equation}
Any KKT point $y^*\in\cK^*$ is 
an optimal solution of \eqref{eq:dualproblem}.
Let $b_k:=b+\frac{y_k}{\mu_k}$ for all $k\geq 1$. For $\kappa>0$, define
\begin{eqnarray*}
(\cP_k):  \quad\lefteqn{\min_{x\in\chi} \left\{p(x)+\kappa~d_{\cK}(Ax-b_k)\right\},}\\
&=& \min_{x\in\reals^n, s\in\cK} \left\{ p(x)
+\mathbf{1}_{\chi}(x)+\kappa~\norm{Ax-b_k-s}_2 \right\},\\
&=& \max_{\norm{w}_2\leq\kappa}\ \min_{x\in\reals^n, s\in\cK} \left\{p(x)
+\mathbf{1}_{\chi}(x)+\fprod{w,~Ax-b_k-s}\right\},\\
&= & \max_{\norm{w}_2\leq\kappa} \left\{ -\fprod{b_k, w} +\inf_{s\in\cK}\fprod{-w,s} -\sup_{x\in\reals^n}
\left\{\fprod{-A^Tw,x}-\left(p(x)+\mathbf{1}_{\chi}(x)\right)\right\}\right\}.
\end{eqnarray*}
Since $\inf_{s\in\cK}\fprod{-w, s} > -\infty$, only if $-w \in
\cK^\ast$; by setting  $y=-w$, we obtain the following dual problem
$(\cD_k)$ of $(\cP_k)$:
\begin{equation*}
(D_k): \quad \max_{\norm{y}_2\leq\tau,~y\in\cK^*}
\left\{\fprod{b_k,y}-(p+\mathbf{1}_\chi)^*(A^Ty)\right\}.
\end{equation*}
Since $y^*\in\cK^*$ is feasible to $(\cD_k)$ for
$\kappa=\norm{y^*}_2$, and  $x_k\in\chi$ is feasible to $(P_k)$, weak
duality implies that
\eq
p(x_k)+\norm{y^*}_2~d_\cK(Ax_k-b_k)\geq
\fprod{b,y^*}-(p+\mathbf{1}_\chi)^*(A^Ty^*)+\frac{1}{\mu_k}\fprod{y_k,
  y^*} =p^*+\frac{1}{\mu_k}\fprod{y_k, y^*},
\en
where the equality follows from strong duality between $(P)$ and $(D)$.
\end{proof}

Next, we upper bound the suboptimality.
\begin{theorem}
\label{thm:subopt-upper}
Let $\{x_k,y_k\}\in\chi\times\cK^*$ be the sequence of primal-dual
ALCC  iterates corresponding to a parameter sequence
$\{\alpha_k, \eta_k, \mu_k\}$ satisfying
\eqref{eq:params}.  Let $p^*$ denote the optimal value of $(P)$.
Then
\begin{equation}
\label{eq:subopt-upper}
P_k(x_k,
y_k) -
\frac{1}{\mu_k}~p^*\leq\frac{1}{\mu_k}
~\xikstar+\frac{1}{2\mu_k^2}~\norm{y_k}_2^2,
\end{equation}
where $\xikstar = \max\{\alpha_k, \eta_k~d_{\chi^*}(x_k)\}$
and $\chi^*$ denote the set of optimal solutions to $(P)$.
\end{theorem}
\begin{proof}
Fix $k\geq 1$ and let $x^*\in\chi^*$. 
Suppose that $x_k = \oracle(P_k, y_k, \alpha_k, \eta_k, \mu_k)$
satisfies \eqref{eq:eps_opt}. Then, since $x^*\in\chi$, from
\eqref{eq:inexact-subopt}, it follows that
\begin{equation}
P_k(x_k,y_k)\leq \inf_{x\in\chi}P_k(x, y_k)+\frac{\alpha_k}{\mu_k}
\leq P_k(x^*, y_k)+\frac{\alpha_k}{\mu_k}. \label{eq:subopt-upper-inexact}
\end{equation}
Next, suppose that $x_k = \oracle(P_k, y_k, \alpha_k, \eta_k, \mu_k)$
satisfies \eqref{eq:grad_opt}. Then, since
$P_k(x,y_k)+\mathbf{1}_{\chi}(x)$ is convex in $x$ for all $k\geq 1$, it follows that
\begin{equation}
\label{eq:subopt-upper-grad}
P_k(x_k,y_k)\leq P_k(x^*,y_k)+\fprod{q_k, x_k-x^*}\leq
P_k(x^*,y_k)+\frac{\eta_k~\norm{x_k-x^*}_2}{\mu_k}.
\end{equation}
From \eqref{eq:subopt-upper-inexact} and \eqref{eq:subopt-upper-grad},
it follows that
\begin{equation}
P_k(x_k,y_k)-\frac{1}{\mu_k}~p^*\leq
\frac{1}{2}
~d_\cK\left(Ax^*-b-\frac{y_k}{\mu_k}\right)^2
+\frac{\max\{\alpha_k,~\eta_k~\norm{x_k-x^*}_2\}}{\mu_k}. \label{eq:subopt-upper-crude}
\end{equation}
Since $Ax^*-b\in\cK$, Lemma~\ref{lem:infeasibility} implies that $d_\cK\left(Ax^*-b-\frac{y_k}{\mu_k}\right)\leq\frac{\norm{y_k}_2}{\mu_k}$. Moreover, since $x^*\in\chi^*$ is arbitrary, from \eqref{eq:subopt-upper-crude} it follows that
\begin{equation}
P_k(x_k,y_k)-\frac{1}{\mu_k}~p^*\leq\frac{\norm{y_k}^2_2}{2\mu_k}
+\frac{\max\{\alpha_k,~\eta_k~\inf_{x^*\in\chi^*}\norm{x_k-x^*}_2\}}{\mu_k}. \label{eq:subopt-upper1}
\end{equation}
\end{proof}

Note that since $f_k(\cdot)\geq 0$,  we have
$P_k(x_k,y_k)\geq\frac{1}{\mu_k}~p(x_k)$ for all $k\geq 1$. Hence,
\begin{equation}
\label{eq:subopt-upper2}
p(x_k)-p^*\leq \xikstar+\frac{1}{2\mu_k}~\norm{y_k}_2^2.
\end{equation}
Now, we establish a bound on the infeasibility of the primal ALCC iterate sequence.
\begin{theorem}
Let $\{x_k,y_k\}\in\chi\times\cK^*$ denote the sequence of primal-dual
ALCC iterates for a parameter sequence
$\{\alpha_k, \eta_k, \mu_k\}$ satisfying \eqref{eq:params} and $y^*\in\cK^*$ be a KKT point of $(P)$. Then
\begin{equation}
\label{eq:penalty}
0\leq d_\cK\left(Ax_k-b\right)\leq \frac{\norm{y_k}_2+\norm{y_{k+1}-y_k}_2}{\mu_k}
\end{equation}
for all $k\geq 1$, where $\xikstar = \max\{\alpha_k, \eta_k~d_{\chi^*}(x_k)\}$
and $\chi^*$ denote the set of optimal solutions to $(P)$.
\end{theorem}
\begin{proof}
From Step~\ref{algeq:y_update} in \alcc, it follows that
\begin{align*}
\frac{y_{k+1}-y_k}{\mu_k}&=\Pi_\cK\left(Ax_k-b-\frac{y_k}{\mu_k}\right)-(Ax_k-b),\\
&=\Pi_\cK\left(Ax_k-b-\frac{y_k}{\mu_k}\right)-\Pi_\cK(Ax_k-b)+\Pi_\cK(Ax_k-b)-(Ax_k-b).
\end{align*}
Hence,
\begin{equation*}
d_\cK(Ax_k-b)\leq\frac{\norm{y_{k+1}-y_k}_2}{\mu_k}+\left\|\Pi_\cK\left(Ax_k-b-\frac{y_k}{\mu_k}\right)-\Pi_\cK(Ax_k-b)\right\|_2.
\end{equation*}
The result now follows from the fact that $\Pi_\cK$ is non-expansive.
\end{proof}

In the next theorem we establish the convergence rate of \alcc.
\begin{theorem}
Let $\{x_k,y_k\}\in\chi\times\cK^*$ denote the sequence of primal-dual
ALCC iterates for a parameter sequence
$\{\alpha_k, \eta_k, \mu_k\}$ satisfying \eqref{eq:mult-seq}. Then for
all $\epsilon>0$, $d_\cK(Ax_k-b)\leq\epsilon$ and
$|p(x_k)-p^*|\leq\epsilon$ within
$\cO\left(\log\left(\epsilon^{-1}\right)\right)$ \oracle\ calls, which
require solving at most
$\cO\left(\epsilon^{-1}\log\left(\epsilon^{-1}\right)\right)$ problems
of the form \eqref{eq:nonsmooth-operation}.
\end{theorem}
\begin{proof}
To simplify the notation, let $\alpha_0=\eta_0=\mu_0=1$, and, without
loss of generality, assume that $1\leq \cD$, where $\cD:=\max_{x\in\chi}d_{\chi^*}(x)\leq\Delta_\chi<\infty$. Then, clearly $d_{\chi^*}(x_k)\leq \cD$ for all $k\geq 1$.

First, \eqref{eq:penalty} implies that
\begin{equation}
\label{eq:infeasibility}
d_\cK(Ax_k-b)\leq \frac{1}{\beta^k}\left(\norm{y_k}_2+\norm{y_{k+1}-y_k}_2\right).
\end{equation}
Moreover, from Step~\ref{algeq:y_update} of \alcc, it follows that
\begin{equation}
\label{eq:penalty-explicit}
d_\cK\left(Ax_k-b-\frac{y_k}{\mu_k}\right)\leq \frac{\norm{y_{k+1}}_2}{\mu_k}=
\frac{1}{\beta^k}\norm{y_{k+1}}_2.
\end{equation}
Now, Theorem~\ref{thm:subopt-lower}, \eqref{eq:subopt-upper2} and \eqref{eq:penalty-explicit} together imply that
\begin{equation}
\label{eq:subopt}
|p(x_k)-p^*|\leq
\frac{1}{\beta^k}\max\left\{\norm{y^*}_2\left(\norm{y_{k+1}}_2+\norm{y_k}_2\right),
~\frac{\cD}{k^{2(1+c)}}+\frac{\norm{y_k}_2^2}{2}\right\}
\end{equation}
Theorem~\ref{thm:bounded-y} shows that $\{y_k\}$ is a bounded
sequence. Hence, from \eqref{eq:infeasibility} and \eqref{eq:subopt},
we have
\begin{equation}
\label{eq:outer-iter-complexity}
d_\cK(Ax_k-b)=\cO\left(\frac{1}{\beta^k}\right), \quad \quad
|p(x_k)-p^*|=\cO\left(\frac{1}{\beta^k}\right).
\end{equation}
Hence, \eqref{eq:outer-iter-complexity} implies that for all $\epsilon>0$,
an $\epsilon$-optimal and $\epsilon$-feasible solution to
$(P)$ can be computed within $\cO\left(\log\left(\epsilon^{-1}\right)\right)$ iterations of \alcc.

The values of $L_{\bar{\gamma}_k}$, $\alpha_k$ and $\mu_k$ are given respectively in \eqref{eq:lipschitz-constant} and \eqref{eq:mult-seq}. Substituting them in the expression for $\ell_{\max}(k)$ in \eqref{eq:lmax} and using the fact that $d_{\chi_k^*}(x_{k-1})\leq\Delta_\chi$, we obtain
\begin{equation}
\label{eq:inner-iter-complexity}
\ell_{\max}(k)\leq\sqrt{\frac{2L_\gamma}{\beta^k}+2\sigma^2_{\max}(A)}~d_{\chi_k^*}(x_{k-1})~\beta^k k^{1+c}=\cO\left(\beta^k k^{1+c}\right).
\end{equation}
Hence, \eqref{eq:inner-iter-complexity} imply that at most $\cO\left(\epsilon^{-1} \log(\epsilon^{-1})\right)$
problems of the form \eqref{eq:nonsmooth-operation} are solved during $\cO\left(\log\left(\epsilon^{-1}\right)\right)$ iterations of \alcc. Indeed, let $N_\epsilon\in\integers_+$ denote total number of problems of the form \eqref{eq:nonsmooth-operation} solved to compute an $\epsilon$-optimal and $\epsilon$-feasible solution to $(P)$. From \eqref{eq:outer-iter-complexity} and \eqref{eq:inner-iter-complexity}, it follows that there exists $c_1>0$ and $c_2>0$ such that
\begin{equation*}
N_\epsilon\leq\sum_{k=1}^{\log_\beta\left(\frac{c_1}{\epsilon}\right)}\ell_{\max}(k)\leq\sum_{k=1}^{\log_\beta\left(\frac{c_1}{\epsilon}\right)}c_2\beta^k k^{1+c}\leq \frac{\beta}{\beta-1}\left(\frac{c_1}{\epsilon}-1\right)\left(\log_\beta\left(\frac{c_1}{\epsilon}\right)\right)^{1+c}.
\end{equation*}
\end{proof}
\begin{corollary}
\label{cor:opt}
Let $\{x_k,y_k\}\in\chi\times\cK^*$ denote the sequence of
primal-dual ALCC iterates for a parameter sequence $\{\alpha_k,~\eta_k,~\mu_k\}$ satisfying \eqref{eq:mult-seq}. Then $\lim_{k\in\integers_+}p(x_k)=p^*$
and $\lim_{k\in\integers_+}d_\cK(Ax_k-b)=0$. Moreover, for all $\cS\subset\integers_+$ such that $\bar{x}=\lim_{k\in\cS}x_k$, $\bar{x}$ is an optimal solution to $(P)$.
\end{corollary}
\begin{proof}
Since $\chi$ is compact, 
Bolzano--Weierstrass theorem implies that there exists a subsequence
$\cS\subset\integers_+$ such that $\bar{x}=\lim_{k\in\cS}x_k$ exists. Moreover,
taking the limit of both sides of \eqref{eq:infeasibility} and
\eqref{eq:subopt}, we have $\lim_{k\in\integers_+}d_\cK(Ax_k-b)=0$ and $\lim_{k\in\integers_+}p(x_k)=p^*$.  Hence, $\lim_{k\in\cS}d_\cK(Ax_k-b)=0$ and $\lim_{k\in\cS}p(x_k)=p^*$.
\end{proof}\\
Note that even though $p(x_k) \rightarrow p^\ast$, the primal iterates themselves may not converge.

Rockafellar~\cite{Roc73_1J} proved that the dual iterate sequence
$\{y_k\}$ computed via
\eqref{eq:x-update_rock}--\eqref{eq:y-update_rock}, converges to a KKT
point of \eqref{eq:rockafellar_problem}. We want to extend this
result to the case where $\cK$ is a general convex cone.
The
proof in \cite{Roc73_1J} 
uses the fact that the penalty multiplier $\mu$
is fixed in \eqref{eq:x-update_rock}--\eqref{eq:y-update_rock} and it
is not immediately clear how to extend this result to the setting with $\{\mu_k\}$ such that
$\mu_k\rightarrow\infty$. In Theorem~\ref{thm:dual-limit}, we extend
Rockafellar's result in \cite{Roc73_1J} 
to arbitrary convex cones $\cK$ when $f(x)=Ax-b$ and the penalty multipliers $\mu_k\rightarrow\infty$. After we independently proved~Theorem~\ref{thm:dual-limit}, we became
  aware of an earlier work of Rockafellar~\cite{Rockafellar1976}
  where he also  extends the dual convergence result in \cite{Roc73_1J} to
  the setting where $\{\mu_k\}$ is an increasing sequence. See
  Section~\ref{sec:previous} for  a detailed
  discussion of our contribution in relation to this earlier work by
  Rockafellar.
\begin{theorem}
\label{thm:dual-limit}
Let $\{x_k,y_k\}\in\chi\times\cK^*$ denote the sequence of primal-dual
ALCC iterates corresponding to a parameter sequence
$\{\alpha_k, \eta_k, \mu_k\}$ satisfying \eqref{eq:mult-seq}. 
Then $\bar{y}:=\lim_{k\in\integers_+}y_k$ exists and $\bar{y}$ is a KKT
point of $(P)$ in \eqref{eq:conic_problem}.
\end{theorem}
\begin{proof}
It follows from~\eqref{eq:y-My} that for all $k\geq 1$ we have
\begin{equation}
\label{eq:y-My-limit}
\lim_{k\in\integers_+}\norm{y_{k+1}-\pi_{\mu_k}(y_k)}_2\leq\lim_{k\in\integers_+}
\sqrt{2\xik\mu_k}=0,
\end{equation}
where $\xik$ is defined in \eqref{eq:xik}. Moreover,
Theorem~\ref{thm:bounded-y} shows that $\{y_k\}$ is a
bounded sequence. Hence, \eqref{eq:y-My-limit} implies that
$\{\pi_{\mu_k}(y_k)\}$ is also a bounded sequence.

From \eqref{eq:g_moreau}, it follows that
$g_{\mu_k}(y_k)=g_0(\pi_{\mu_k}(y_k))-\frac{1}{2\mu_k}~\norm{\pi_{\mu_k}(y_k)-y_k}_2^2$
and $g_{\mu_k}(y_k)\geq g_0(y^*)-\frac{1}{2\mu_k}~\norm{y^*-y_k}_2^2$
for any KKT point $y^*$. Since $g_0(y^*)=p^*$, we have that
\begin{equation}
\label{eq:g0-lowerbound}
g_0(\pi_{\mu_k}(y_k))\geq p^*-\frac{1}{2\mu_k}~\norm{y^*-y_k}_2^2.
\end{equation}
Since $\{y_k\}$ is bounded, taking the limit inferior of both sides of
\eqref{eq:g0-lowerbound} we obtain
\begin{equation}
\liminf_{k\in\integers_+}g_0(\pi_{\mu_k}(y_k))\geq p^*-\lim_{k\in\integers_+}\frac{1}{2\mu_k}~\norm{y^*-y_k}_2^2=p^*.\label{eq:liminf}
\end{equation}
Moreover, since $\pi_{\mu_k}(y_k)\in\cK^*$ for all $k\geq 1$, weak
duality implies that $\limsup_{k\in\integers_+}g_0(\pi_{\mu_k}(y_k))\leq
p^*$. Thus, using \eqref{eq:liminf}, we have that
\begin{equation}
\label{eq:lim-g0}
\lim_{k\in\integers_+}g_0(\pi_{\mu_k}(y_k))=p^*.
\end{equation}
Since $\{\pi_{\mu_k}(y_k)\}$ is bounded, there exists
$\cS\subset\integers_+$ and $\bar{y}\in\cK^*$ such that
\begin{equation}
\label{eq:y-lim}
\bar{y}:=\lim_{k\in\cS}\pi_{\mu_k}(y_k)=\lim_{k\in\cS}y_{k+1},
\end{equation}
where the last equality follows from \eqref{eq:y-My-limit}.

From \eqref{eq:Lmu_def} and \eqref{eq:g_mu}, it follows that
\begin{equation*}
g_0(y)=\inf_{x\in\chi,s\in\cK}\big\{p(x)-\fprod{y,Ax-s-b}\big\}.
\end{equation*}
Hence, $-g_0$ is a pointwise supremum of linear functions, which are
always closed. Lemma~3.1.11 in \cite{Nesterov04} establishes that $-g_0$ is
a closed convex function. Since a closed convex function is always
lower semicontinous, we can conclude that $-g_0$ is lower
semicontinuous, or equivalently,  $g_0$ is
an upper semicontinuous function. Hence, \eqref{eq:lim-g0} and
\eqref{eq:y-lim} imply that
\begin{equation*}
p^*=\lim_{k\in\integers_+}g_0(\pi_{\mu_k}(y_k))=\limsup_{k\in\cS}g_0(\pi_{\mu_k}(y_k))\leq
g_0(\bar{y})\leq p^*,
\end{equation*}
where the first inequality is due to upper semicontinuity of $g_0$ and
the last one is due to weak duality and the fact that $\bar{y}\in\cK^*$. Thus, we have
\begin{equation}
\label{eq:KKT-y}
g_0(\bar{y})=\lim_{k\in\integers_+}g_0(\pi_{\mu_k}(y_k))=p^*,
\end{equation}
which implies that $\bar{y}\in\cK^*$ is a KKT point of \eqref{eq:conic_problem}.

Moreover, since \eqref{eq:y-induction-eps} holds for any KKT point, we can
substitute $\bar{y}$ for $y^*$ in the expression. Thus, we have
\begin{equation}
\label{eq:convergent-y}
\norm{y_{\ell}-\bar{y}}_2\leq\norm{y_k-\bar{y}}_2+\sum_{t\geq
  k}\sqrt{2\xi_t\mu_t},\quad \forall \ell> k.
\end{equation}
Fix $\epsilon>0$. Since the sequence $\{\sqrt{\xik\mu_k}\}$ is summable, it follows
that there exists $N_1\in\integers_+$ such that
$\sum_{t=k}^{\infty}\sqrt{2\xi_t\mu_t}\leq\frac{\epsilon}{2}$ for all $k > N_1$. Moreover, since  the $\{y_k\}_{k \in \cS}$  converges to $\bar{y}$, it follows that there exists $N_2 \in\cS$ such that $N_2\geq N_1$ and $\norm{y_{N_2}-\bar{y}}_2\leq\frac{\epsilon}{2}$.
Hence, \eqref{eq:convergent-y} implies that
$\norm{y_\ell-\bar{y}}_2\leq\epsilon$ for all $\ell> N_2$. Therefore, $\lim_{k\in\integers_+}y_k=\bar{y}$.
\end{proof}
\section{Conclusion}
In this paper we build on previously known augmented Lagrangian
algorithms for convex
problems with standard inequality
constraints~\cite{Rockafellar-73,Roc73_1J} to develop the \alcc\ that
 solves convex problems with conic constraints.
 In each iteration of the \alcc, a sequence of ``penalty'' Lagrangians---see \eqref{eq:L_mu}---are
inexactly minimized  
over a ``simple'' closed convex set.
We show that 
recent results on optimal first-order
algorithms~\cite{Beck09_1J,Tseng08}~(see also \cite{Nesterov04,Nesterov05}), can be used to
 bound the number of basic operations needed in each iteration to
inexactly minimize the ``penalty'' Lagrangian sub-problem.
By carefully controlling the growth of the penalty parameter $\mu_k$ that
controls the iteration complexity of \alcc, and the decay of
parameter $\alpha_k$ that controls the suboptimality of each
sub-problem, 
we show that \alcc\ is a theoretically efficient
first-order, inexact augmented Lagrangian algorithm for structured
non-smooth conic convex programming.

\bibliographystyle{siam}
\bibliography{All}
\appendix
\section{Proofs of technical results}
\label{app:proofs}
\begin{lemma}
\label{lem:dist}
Let $f(\cdot)=\frac{1}{2}d^2_\cK(\cdot)$. Then $f$ is convex, and $\grad
f(y)=y-\Pi_\cK(y)$ 
is Lipschitz
continuous with Lipschitz constant equal to 1. Moreover, both
$\Pi_\cK(\cdot)$ and $\Pi^c_\chi(z) = z -\Pi_\cK(z)$ are
nonexpansive. 
\end{lemma}
\begin{proof}
The indicator function $\ones_{\cK}(\cdot)$  of a closed convex set
$\cK$ is a proper closed convex function, and
\eq
f(y)=
\min_{z\in\reals^m}\{\ones_{\cK}(\cdot)(z)+\frac{1}{2}\norm{z-y}_2^2\}=
\min_{z\in\cK}\frac{1}{2}\norm{z-y}_2^2,
\en
is the Moreau regularization of the function  $\ones_{\cK}(\cdot)$,
and the projection operator $\Pi_\cK(\cdot)$ is the corresponding
Moreau proximal map. Therefore, all the results of
this lemma follow from Lemma~\ref{lem:prox}.
\end{proof}
\newpage
\begin{lemma}
\label{lem:infeasibility}
For all $y,y'\in\reals^m$, $d_\cK(y)\leq d_\cK(y+y')+\norm{y'}_2$.
\end{lemma}
\begin{proof}
\begin{eqnarray*}
d_\cK(y)&= &\norm{\Pi_\cK(y)-y}_2=\norm{\Pi_\cK(y)-y
  +\Pi_\cK(y+y')-\Pi_\cK(y+y')+y'-y'+y-y}_2,\\
&\leq &\norm{\Pi_\cK(y+y')-(y+y')}_2
+\norm{\Pi^c_\cK(y+y')-\Pi^c_\cK(y)}_2,\\
&\leq &d_\cK(y+y')+\norm{y'}_2,
\end{eqnarray*}
where the last inequality follows from the fact that
$\Pi^c_\chi(x) = x -\Pi_\cK(x)$ is nonexpansive.
\end{proof}
\subsection*{Proof of Lemma~\ref{lem:lipschitz}}
$\mbox{}$
\begin{proof}
For all $y\in\reals^m$, the convexity of $f_k(x,y)$ in $x$ follows from Lemma~\ref{lem:dist}.

Moreover, Lemma~\ref{lem:dist} and the chain rule,  together imply
\eqref{eq:gradfk}. Now, fix $x', x''\in\reals^n$ and
$\bar{y}\in\reals^m$. Then
\eqref{eq:gradfk} implies that
\begin{align*}
\lefteqn{\norm{\grad_x f_k(x',\bar{y})-\grad_x f_k(x'',\bar{y})}_2}\\
&= \left\|A^T\left[Ax'-b-\frac{\bar{y}}{\mu_k}-\Pi_\cK\left(Ax'-b-\frac{\bar{y}}{\mu_k}\right)-
    \left(Ax''-b-\frac{\bar{y}}{\mu_k}-\Pi_\cK\left(Ax''-b-\frac{\bar{y}}{\mu_k}\right)\right)\right]\right\|_2,\\
&\leq \sigma_{\max}(A)\norm{A(x'-x'')}_2 \leq  \sigma^2_{\max}(A)~\norm{x'-x''}_2,
\end{align*}
where the first inequality follows from the non-expansiveness of $\Pi^c_{\chi}(\cdot)$.
\end{proof}
\subsection*{Proof of Lemma~\ref{lem:cone}}
$\mbox{}$
\begin{proof}
$\Pi_\cK(x)\in\argmin_{s\in\cK}\norm{s-x}_2^2$, if, and only if,
$\fprod{\Pi_\cK(x)-x,~s-\Pi_\cK(x)}\geq 0$ for all $s\in\cK$. Hence,
\begin{equation}
\label{eq:ap1}
\fprod{\Pi_\cK(x)-x,~s}\geq \fprod{\Pi_\cK(x)-x,~\Pi_\cK(x)}, \quad \quad \forall s\in\cK.
\end{equation}
Since the left hand side of \eqref{eq:ap1} is bounded from below for
all $s\in\cK$, it follows that $\Pi_\cK(x)-x\in\cK^*$. Moreover, since
$\Pi_\cK(x)\in\cK$, we have
\begin{equation*}
0=\min_{s\in\cK}\fprod{\Pi_\cK(x)-x,~s}\geq\fprod{\Pi_\cK(x)-x,~\Pi_\cK(x)}\geq 0.
\end{equation*}
This implies $\fprod{\Pi_\cK(x)-x,~\Pi_\cK(x)}=0$.

Suppose $x\in-\cK$.  
Clearly, $\fprod{\bo -x,~s - \bo}\geq 0$ for all $s\in\cK$. Thus, it
follows that $\Pi_{\cK}(x) = \bo$.
\end{proof}
\end{document}